\documentclass[12pt, a4paper]{article}
 \usepackage{amsmath}
 \usepackage{amssymb}
\usepackage[margin=21mm]{geometry}
\usepackage[utf8]{inputenc}
\usepackage{graphics}
\usepackage{xspace}
\usepackage{color}
 \usepackage{amsthm}
 \usepackage[old]{old-arrows}
\usepackage[dvipsnames]{xcolor}
\usepackage{mathtools}
\usepackage{latexsym}
\newcommand{\R}{{\mathbb R}}

\newcommand{\dyle}{\displaystyle}
\newcommand{\dint}{\dyle\int}

\newtheorem{Theorem}{Theorem}[section]
\numberwithin{Theorem}{section}
\newtheorem{Corollary}[Theorem]{Corollary}
\newtheorem{Lemma}[Theorem]{Lemma}
\newtheorem{Proposition}[Theorem]{Proposition}
\newtheorem{Definition}[Theorem]{Definition}

\newtheorem{remark}[Theorem]{Remark}
\usepackage[mathscr]{euscript}
\usepackage{mathrsfs}
 \usepackage{enumerate}
 \usepackage{cite}
 \usepackage{tikz}
 \usepackage[pagebackref]{hyperref}
 \usepackage{hyperref}

\numberwithin{equation}{subsection}

\catcode`@=12

\makeatother
\usepackage[english]{babel}

 \usepackage[hyperpageref]{backref}
\title{Spectrum properties of mixed operators under the mixed boundary conditions}
\date{}
\author{Lovelesh Sharma$^{1}\thanks{Corresponding author(sharma.94@iitj.ac.in)}$    \\
       \small $^{1}$ Department of Mathematics, Indian Institute of Technology Jodhpur, Rajasthan 342030, India \\
}

\providecommand{\keywords}[1]
{
  \small	
  \textbf{\textit{Keywords---}} #1
}
\begin{document}
\maketitle \vspace{-1.8\baselineskip}
\begin{abstract}
In this paper, we describe the spectrum properties of  mixed operators, precisely the superposition of the classical Laplace operator and the fractional Laplace operator in the presence of mixed  boundary conditions, that is 
 \begin{equation}  \label{1}
    \left\{\begin{split} \mathcal{L}u\: &= \lambda u,~~\text{in} ~\Omega, \\
      u&=0~~~~~\text{in} ~~{U^c},\\
 \mathcal{N}_s(u)&=0 ~~~~~\text{in} ~~{\mathcal{N}}, \\
 \frac{\partial u}{\partial \nu}&=0 ~~~~~\text{in}~~ \partial \Omega \cap \overline{\mathcal{N}},
    \end{split} \right.\tag{$P_\lambda$}
\end{equation}
where  $U= (\Omega \cup {\mathcal{N}} \cup (\partial\Omega\cap\overline{\mathcal{N}}))$, $\Omega \subseteq \mathbb{R}^n$ is a non empty bounded open
set with sufficiently smooth boundary $\partial\Omega$, say of class $C^1$, and $\mathcal{D}$, $\mathcal{N}$ are open subsets of $\mathbb{R}^n\setminus{\bar{\Omega }}$ such that $\overline{{\mathcal{D}} \cup {\mathcal{N}}}= \mathbb{R}^n\setminus{\Omega}$, $\mathcal{D} \cap {\mathcal{N}}=  \emptyset $ and $\Omega\cup \mathcal{N}$ is a bounded set with sufficiently smooth boundary, $\lambda >0$ is a  real parameter and
 $\mathcal{L}=  -\Delta+(-\Delta)^{s},~ \text{for}~s \in  (0, 1).$ 
\end{abstract}

\keywords {Mixed local-nonlocal operators, mixed boundary conditions, principal eigenvalue and eigenfunction.}\\

\textbf{Mathematics Subject Classification:} 47A75, 35J25, 35J20.
\section{Introduction}
We investigate the existence and classical properties of the eigenvalues and eigenfunctions to the following problem
 \begin{equation*} 
\left\{\begin{split} \mathcal{L}u\: &= \lambda u,~~ \text{in} ~\Omega, \\
      u&=0~~~~~\text{in} ~~{U^c},\\
 \mathcal{N}_s(u)&=0 ~~~~~\text{in} ~~{\mathcal{N}}, \\
 \frac{\partial u}{\partial \nu}&=0 ~~~~~\text{in}~~ \partial \Omega \cap \overline{\mathcal{N}},
    \end{split} \right.\tag{$P_\lambda$}
         \end{equation*}
  where  $U= (\Omega \cup {\mathcal{N}} \cup (\partial\Omega\cap\overline{\mathcal{N}}))$, $\Omega \subseteq \mathbb{R}^n$ is a non empty bounded open set with
sufficiently smooth boundary $\partial\Omega$ of class $C^1$, and $\mathcal{D}$, $\mathcal{N}$, respectively denoted open Dirichlet and Neumann set, are disjoint open subsets such that $\overline{{\mathcal{D}} \cup {\mathcal{N}}}= \mathbb{R}^n\setminus{\Omega}$
  	and $\Omega\cup \mathcal{N}$ is a bounded set with smooth boundary.   The following figures illustrate possible configurations of the Dirichlet region $\mathcal{D}$ 
and the Neumann region $\mathcal{N}$ with respect to the domain $\Omega$. 
The set $\mathcal{D}$ corresponds to the exterior portion where the  Dirichlet condition 
$u = 0$ is imposed, while $\mathcal{N}$ represents the complementary exterior region where the 
nonlocal Neumann condition $\mathcal{N}_s(u) = 0$ applies. 
On the boundary portion $\partial \Omega \cap \overline{\mathcal{N}}$, the classical Neumann condition 
$\frac{\partial u}{\partial \nu} = 0$ is prescribed. 
Depending on the geometry, the relative size and position of $\mathcal{D}$ and $\mathcal{N}$ may vary, 
as sketched below.
 \begin{center}
\begin{minipage}{0.48\textwidth}
\centering
\scalebox{0.65}{
\begin{tikzpicture}
  \definecolor{lightgrey}{RGB}{200, 200, 200}

  \fill[lightgrey] (0,0) circle (3cm);
  \fill[white] (0,0) circle (1.5cm); 

  \draw[thick] (0,0) circle (3cm); 
  \draw[thick] (0,0) circle (1.5cm); 

  \node at (-2.2cm,1.5cm) {$\Omega$}; 
  \node at (0,0) {$\mathcal{N}$};     
  \node at (3.4cm,0) {$\mathcal{D}$}; 
\end{tikzpicture}
}
\vspace{1em}

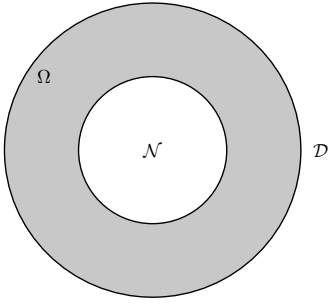
Fig. 1. Annular domain configuration
\end{minipage}
\hfill
\begin{minipage}{0.48\textwidth}
\centering
\scalebox{0.65}{
\begin{tikzpicture}
  \definecolor{lightgrey}{RGB}{200, 200, 200}

  \fill[lightgrey] (-2,0) circle (2cm); 
  \draw[thick] (-2,0) circle (2cm);     
  \node at (-2,0) {$\Omega$};           

  \fill[lightgrey] (2,0) circle (2cm);  
  \draw[thick] (2,0) circle (2cm);      
  \node at (2,0) {$\mathcal{N}$};       

  \node at (0,2.8) {$\mathcal{D}$};  
\end{tikzpicture}
}
\vspace{1em}

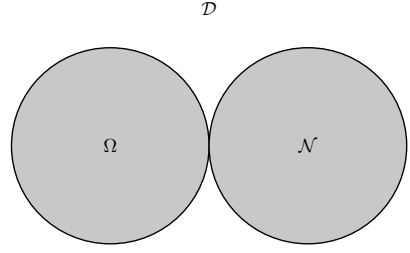
Fig 2.  Disconnected domain configuration
\end{minipage}
\end{center}
 Let $\lambda >0$ is a  real parameter, $\nu$ denotes the outward normal on $\partial\Omega\cap \overline{\mathcal{N}}$ and
 \begin{equation}\label{A}
\mathcal{L}=  -\Delta+(-\Delta)^{s},~ \text{for}~s \in  (0, 1).
 \end{equation}
The term ``mixed" describes an operator that combines local and nonlocal differential operators. In our case, the operator $\mathcal{L}$ in \eqref{1} is generated by the superposition of the classical Laplace operator $-\Delta$ and the fractional Laplace operator $(-\Delta)^{s}$ which is for a fixed parameter $s \in (0,1)$  defined by
$${(- \Delta)^{s}u(x)} = C_{n,s}~ P.V. \int_{\mathbb{R}^n} {\dfrac{u(x)-u(y)}{|x-y|^{n+2s}}} ~ dy. $$ 
The term ``P.V." stands for Cauchy's principal value, while $C_{n,s}$ is a normalizing constant whose explicit expression is given in \cite{dipierro2017nonlocal}.
In the literature, there are numerous definitions of nonlocal normal derivatives. We consider the one given in \cite{dipierro2017nonlocal} and defined for smooth functions $u$ as
\begin{equation}\label{normal}
\mathcal{N}_{s}u(x)= C_{n,s}\dint_{\Omega} \dfrac{u(x)-u(y)}{|x-y|^{n+2s}}\,dy, \qquad  x\in \mathbb{R}^n\setminus\bar{\Omega}.
\end{equation}
  The eigenvalues are often interpreted as vibrating membrane frequencies, as discussed in \cite{Rayleigh}.
The study of mixed operators of the type $\mathcal{L}$ as in problem \eqref{1} is motivated by several applications where such kinds of operators are naturally generated, including the theory of optimal searching, biomathematics, and animal forging, for which we refer to \cite{dipierro2022non}. It is worth mentioning that the operator $\mathcal{L}$ seems to be of interest in biological applications, see, e.g., \cite{dipierro2021description} and the references therein.
For instance, Dipierro and Valdinoci \cite{dipierro2021description} proposed a model for ecological niches where the population density evolves according to a mixed local-nonlocal operator, combining classical and fractional Laplacians. This model effectively captures the dynamics of species dispersal, accounting for both local random walks and long-range jumps, and is governed by a new type of Neumann condition arising from the superposition of Brownian and Lévy processes. From a purely mathematical perspective, the superposition of such operators introduces a lack of scale invariance, which may lead to unexpected
complications. Ôtani et al. \cite{Otani} studied the first eigenvalue and its properties in detail. Lindgren et al. \cite{Lindgreen} analyzed viscosity solutions and highlighted some distinct behaviors of eigenvalues in the nonlocal setting. Valdinoci et al. \cite{SeR} proved the existence of nontrivial solutions for equations involving nonlocal operators with Dirichlet conditions and explored related spectral properties. A solid framework for nonlocal Neumann problems was developed by Audrito et al. \cite{Audrito2023}, who established boundary regularity and introduced tools for weak formulations in bounded domains. For higher-order cases, Barrios et al. \cite{Barrios2020} proposed variational formulations and integration-by-parts formulas for fractional Laplacians with $s > 1$, making it possible to handle Neumann-type conditions in a variational setting.  In \cite{Mugnai2019}, the authors introduced the notion of Neumann fractional  $p$-derivatives, proved the existence of eigenvalues, and analyzed an associated evolution problem.  
Recent work on mixed local–nonlocal models has added valuable insights to the field. Dipierro et. al in \cite{MR4438596} was one the first among the others who consider mixed operator problems in the presence of classical as well as non-local  Neumann boundary conditions. Their recent article  discusses the spectral properties and the $L^\infty$
bounds associated with a mixed local and nonlocal problem,  in relation to some physical motivations arising from
population dynamics and mathematical biology and also
in \cite{Dipierro2025} authors  examined nonlinear equations involving both local and nonlocal operators with Neumann boundary conditions. Using variational methods like the Mountain Pass and Linking Theorems, they performed a spectral analysis and proved existence results. Moreover, authors \cite{Amundsen, AR, Cowan} studied diffusion problems with mixed operators under Neumann conditions, focusing on the existence and regularity of weak solutions.
Recently, Brezis Nirenberg type existence results for mixed operators under mixed boundary conditions have been studied in \cite{LS}, where the first eigenvalue of the associated linear operator play a crucial role. Our present work, which establishes the spectral framework for such operators, thus provides a foundation for further investigations into semilinear problems of Brezis–Nirenberg type problems.  In a related study, Mukherjee et al. \cite{TMLS} analyzed elliptic problems with mixed operators subject to Dirichlet Neumann boundary conditions, offering new insights into variational formulations and solution behavior.

Recently, authors in \cite{pezzo2019eigenvalues} proved the existence of the first eigenvalue of (the mixed type operators), studied its characteristics, and explored the properties of the related eigenfunction. In \cite{Bs}, the authors studied a class of non-degenerate elliptic operators formed by combining the classical Laplacian with a general nonlocal operator. They established results on existence and uniqueness for Dirichlet boundary value problems, as well as maximum principles and generalized eigenvalue problems. As applications of their analysis, they derived a Faber–Krahn-type inequality and a one-dimensional symmetry result related to Gibbon's conjecture.

In addition, Maione et al. \cite{Maione} studied the eigenvalues and eigenfunctions of the operator $\mathcal{L}_{\alpha}$, given by $\mathcal{L}_{\alpha} = -\Delta + \alpha (\Delta)^s$, $\alpha \in \mathbb{R}$, under Dirichlet boundary conditions. In particular, the study of principal eigenvalues is essential in the investigation of non-sign changing solutions to semi-linear problems and in local bifurcation  phenomena as well as in stability analysis  (see \cite{berestycki2016definition}).  
 Next, we recall some eigenvalue problems involving mixed boundary conditions. In \cite{1denzler1998bounds, denzler1999windows}, Denzler et al. studied the an eigenvalue problem involving $\Delta$ with mixed Dirichlet and Neumann boundary conditions  
\begin{equation}\label{local}
   \left\{\begin{split}
   -\Delta u &= \lambda_1(\mathcal{D}) u, \quad u > 0 \quad \text{in } \Omega,\\ 
   u &= 0 ~~\text{on}~~ D,\\
   \frac{\partial u}{\partial \nu} &= 0 ~~\text{on}~~ N,
   \end{split}\right.
\end{equation}  
where they examined the behaviour of the eigenvalue $\lambda_1(\mathcal{D})$ as the sets with Dirichlet or Neumann boundary conditions are varied. Leonori et al. \cite{MR3784437} extended this study for nonlocal operator $(-\Delta)^s$ under mixed boundary conditions. Due to the nonlocal nature of their problem, the sets $D$ and $N$ could have infinite measures, which distinguishes them significantly from the local case. Furthermore, motivated by the aforementioned works, we have recently studied eigenvalue problem for $\mathcal L$ under mixed boundary conditions in \cite{JTL}. 

Inspired by the work in \cite{SeR}, we aim to study spectral and variational aspects of mixed local–nonlocal operators under mixed boundary conditions, where Dirichlet and Neumann-type conditions are applied on disjoint parts of the boundary. Our approach modifies and generalizes the techniques from \cite{SeR} to handle this more versatile setting. Therefore, beyond extending classical spectral results, our study highlights genuinely new aspects. In contrast to the classical spectral theory, the interplay of local and nonlocal effects under mixed boundary conditions creates novel interactions across the boundary parts, leading to new spectral phenomena and technical difficulties absent in the standard frameworks. This present work extends our results to derive classical properties of the spectrum of mixed operators under mixed
boundary conditions.

Such spectral problems are not only of intrinsic mathematical interest but also arise naturally in various applications. For instance, in nonlinear critical problems, as studied in \cite{LS}, the presence of mixed local–nonlocal operators plays a decisive role in the existence and multiplicity of solutions. Moreover, in models of ecology and population dynamics, mixed dispersal mechanisms combining short range Brownian motion with long range Lévy flights lead precisely to operators of the type considered here. Similarly, in anomalous diffusion phenomena, the interplay between local and nonlocal effects captures realistic transport behaviors that cannot be explained by classical models, we refer to \cite{dipierro2021description, dipierro2022non} and and the references therein.

\vspace{0.5em}
We now state the main results of this paper:

\begin{Theorem}\label{prop3.1}
\begin{itemize}
    \item[(1)]
$\lambda_1$ is the first eigenvalue of \eqref{1} i.e. if $\lambda$ is an eigenvalue of
\eqref{1} then $\lambda\geq\lambda_1$.
\item[(2)] The first eigenvalue of $\mathcal L$ with mixed boundary conditions, as in \eqref{1}, is positive i.e. $\lambda_1>0$.
\item[(3)]
 The first eigenvalue $\lambda_1$ is simple; i.e. if $u\in \mathcal{X}^{1,2}_{\mathcal{D}}(U)$ is a solution of 
    \begin{equation}
    \int_{\Omega} \nabla u\cdot\nabla \varphi \,dx +\int_{Q} \frac{(u(x)-u(y))(\varphi(x)-\varphi(y))}{|x-y|^{n+2s}} dxdy  = \lambda_1\int_{\Omega} u\varphi\,dx,
\end{equation}
for every $\varphi\in \mathcal{X}^{1,2}_{\mathcal{D}}(U),$ then $u= \zeta e_1$, where $e_1$ is an eigenfunction corresponding to eigenvalue $\lambda_1$ and $\zeta\in\R$.
\end{itemize}
\end{Theorem}
In the following result below, we provide a characterization of the eigenvalues and eigenfunctions of $\mathcal{L}$ with Dirichlet and Neumann boundary conditions.
\begin{Theorem}\label{p2.9}
The following statements holds true :

\begin{itemize}
    \item[(1)] The operator $\mathcal{L}$ admits a divergent, but bounded from below,  set of eigenvalues of \eqref{dd2.2} consists of a sequence $\{\lambda_k\}_{k \in \mathbb{N}}$ satisfying
    \begin{equation}\label{ineql}
    0 < \lambda_1 < \lambda_2 \leq \cdots \leq \lambda_k \leq \lambda_{k+1} \leq \cdots
        \end{equation}
    such that
    \begin{equation}\label{eql206}
   \lambda_k \to +\infty \quad \text{as } k \to +\infty. 
          \end{equation}
          Furthermore, for any \( k \in \mathbb{N} \), the eigenvalues are characterized as
    \begin{equation}\label{eq204}
    \lambda_{k+1} = \min_{\substack{u \in \mathcal{P}_{k+1} \\ \|u\|_{L^2(\Omega)} = 1}}
    \left\{ \int_\Omega |\nabla u|^2 \, dx +  \int_{Q} 
    \frac{|u(x) - u(y)|^2}{|x - y|^{n+2s}} \, dx \, dy \right\} 
        \end{equation}
    or equivalently,
    \begin{equation}\label{eq206}
    \lambda_{k+1} = \min_{u \in \mathcal{P}_{k+1} \setminus \{0\}} 
    \frac{ \int_{\Omega}|\nabla{u}|^2\,dx+ \int_{Q} \frac{|u(x)-u(y)|^{2}}{|x-y|^{n+2s}} \, dx dy}{\int_\Omega |u(x)|^2 \, dx}, 
        \end{equation}
    where
    \begin{equation}\label{eq207}
   \mathcal{P}_{k+1}= \left\{ u \in\mathcal{X}^{1,2}_{\mathcal{D}}(U) : \langle u, e_j \rangle_{\mathcal{X}^{1,2}_{\mathcal{D}}(U)} = 0, \quad \forall~ j = 1, \ldots, k \right\} 
         \end{equation}
    and $e_j$ is the eigenfunction corresponding to eigenvalue $\lambda_j.$
    \item[(2)] For any $k \in \mathbb{N}$, there exists an eigenfunction $e_{k+1} \in \mathcal{P}_{k+1}$ corresponding to eigenvalue  $\lambda_{k+1}$, attaining the minimum in \eqref{eq204}; i.e., $\|e_{k+1}\|_{L^2(\Omega)} = 1$ and
    \begin{equation}\label{eq2011}
  \lambda_{k+1}= \int_\Omega |\nabla e_{k+1}|^2 \, dx +  \int_{Q} 
    \frac{|e_{k+1}(x) - e_{k+1}(y)|^2}{|x - y|^{n+2s}} \, dx \, dy . 
          \end{equation}
\end{itemize}
    \end{Theorem}
     \begin{Theorem}\label{p210}
      The sequence \( \{e_k\}_{k \in \mathbb{N}} \) of eigenfunctions associated with \( \lambda_k \) forms an orthonormal basis for \( L^2(\Omega) \) and an orthogonal basis for \( \mathcal{X}^{1,2}_{\mathcal{D}}(U) \).
 \end{Theorem}

The rest of the paper is organized as follows:
In Section 2, we introduce the functional framework and some preliminaries results.
Section 3 is devoted to the proof of fundamental spectral properties of the mixed operator, including the existence and characterization of eigenvalues and eigenfunctions.
\section{Functional framework and main results }
In this section, we set our notations and formulated the functional framework for \eqref{1}, to be used throughout the paper.
For every $s\in (0,1)$, we recall the fractional Sobolev spaces,
$${H^{s}(\mathbb{R}^n)} =  \Bigg\{ u \in L^{2}(\mathbb{R}^n):~~\frac{|u(x) - u(y)|}{|x - y|^{\frac{n}{2} + s}} \in L^{2}({\mathbb{R}^n}\times {\mathbb{R}^n)} \Bigg\} $$ which contain $H^1(\mathbb R^n)$. We assume that $\Omega \cup \mathcal N$ is bounded with a smooth boundary. 
The symbol $U$  is used throughout the article instead of $(\Omega \cup {\mathcal{N}} \cup (\partial\Omega\cap\overline{\mathcal{N}}))$ for sake of clarity.
We define the function space $\mathcal{X}^{1,2}_{\mathcal{D}}(U)$ as 
\begin{align*}
    \mathcal{X}^{1,2}_{\mathcal{D}}(U)  = \{u\in H^1(\mathbb{R}^n) : ~u|_{U} \in H^1_0(U) ~\text{and}~ u \equiv 0~ a.e. ~\text{in}~ {U^c}\}.
\end{align*}
Let us define 
 $$ \eta(u)^2 =  ||\nabla{u}||^2_{L^{2}(\Omega)}+ [u]^2_{s},$$
for $u\in\mathcal{X}^{1,2}_{\mathcal{D}}(U)$, where  $[u]_s$ is the Gagliardo seminorm of $u$ defined by 
 $$[u]^2_{s} = ~ \bigg(\int_{Q} \frac{|u(x)-u(y)|^{2}}{|x-y|^{n+2s}} \, dx dy \bigg)$$ and $Q= \mathbb R^{2n}\setminus (\Omega^c\times \Omega^c)$.
The following Poincar\'e type inequality can be established following the arguments of Proposition 2.4 in \cite{MR4065090} and taking advantage of partial Dirichlet boundary conditions in $U^c$.
\begin{Proposition}\label{Poin} (Poincar\'e type inequality) There exists a constant $C=C(\Omega, n,s)>0$ such that
$$
\dint_{\Omega}| u|^2\,dx\leq C\bigg(\int_{\Omega} |\nabla u|^2\,dx+  \int_{Q} \dfrac{|u(x)-u(y)|^2}{|x-y|^{n+2s}}\,dxdy\bigg),
$$
for every  $u\in\mathcal{X}^{1,2}_{\mathcal{D}}(U)$, i.e. $\|u\|^2_{L^2(\Omega)}\leq C \eta(u)^2$. 
\end{Proposition}
As a consequence of Proposition \ref{Poin}, $\eta(\cdot)$ forms a norm on $\mathcal{X}^{1,2}_{\mathcal{D}}(U)$ and $\mathcal{X}^{1,2}_{\mathcal{D}}(U)$ is a Hilbert space with the  inner product associated with $\eta(\cdot)$, defined for any $u,v\in \mathcal{X}^{1,2}_{\mathcal{D}}(U)$  by
\begin{equation}\label{eq201}
    {\langle{ u},{ v}\rangle}_{\mathcal{X}^{1,2}_{\mathcal{D}}(U)} = \int_{\Omega} \nabla u. \nabla{v} \,dx + \int_{Q} {\dfrac{(u(x)-u(y)) (v(x)-v(y))}{|x-y|^{n+2s}}} ~ dx dy.
    \end{equation} 

Consequently, we have the integration by-parts formula given in the following proposition.
 \begin{Proposition}\label{P}
 For every $ u,v\in  C^\infty_0(U)$, it holds
\begin{align*}
    \int_{\Omega}v \mathcal{L} u \,dx  
    &= \int_{\Omega} \nabla u \cdot\nabla{v} \,dx +  \int_{Q} {\dfrac{(u(x)-u(y)) (v(x)-v(y))}{|x-y|^{n+2s}}} ~ dx dy-  \int_{\partial \Omega\cap\overline{{\mathcal{N}}}} v {\frac{\partial u}{\partial \nu}}~ d{\sigma}-  \int_{{\mathcal{N}}} v {\mathcal{N}}_s u~ dx.
    \end{align*}
    where $\nu$ denotes the outward normal on $\partial\Omega$.
    \end{Proposition}
    \begin{proof}
        By directly using the integration by parts formula and the fact that $u,v \equiv 0$ a.e. in $\mathcal{D}\cup(\partial\Omega\cap\overline{\mathcal{D}})=U^c$, we can follow [\cite{dipierro2017nonlocal}, Lemma 3.3 ], to obtain the conclusion.
    \end{proof}
  \begin{Corollary}
Since $ C^\infty_0(U)$ is dense in $\mathcal{X}^{1,2}_{\mathcal{D}}(U)$, so Proposition \ref{P} still holds for functions in $\mathcal{X}^{1,2}_{\mathcal{D}}(U)$.   
\end{Corollary}
We now define the notion of weak solution to \eqref{1}, which consists of the below eigenvalue problem \eqref{dd2.2}.
\begin{Definition}\label{d1}
We say that $u\in \mathcal{X}^{1,2}_{\mathcal{D}}(U)$ is a weak solution to \eqref{1} if 
\begin{equation}\label{dd2.2}
    \int_{\Omega} \nabla u\cdot\nabla \varphi \,dx +\int_{Q} \frac{(u(x)-u(y))(\varphi(x)-\varphi(y))}{|x-y|^{n+2s}} dxdy  = \lambda\int_{\Omega} u\varphi\,dx,  \forall \varphi \in \mathcal{X}^{1,2}_{\mathcal{D}}(U).
\end{equation}

\end{Definition}

Consequently to $\mathcal{X}^{1,2}_{\mathcal{D}}(U)\hookrightarrow H^1(\mathbb R^n)$ and Sobolev embeddings, we infer the following embedding result.

\begin{remark}\label{r2.5}
   For $U$ is bounded (since $\Omega\cup\mathcal{N}$ is bounded) with smooth boundary, then we have compact embedding
    $$\mathcal{X}^{1,2}_{\mathcal{D}}(U)\hookrightarrow \hookrightarrow L^q_{loc}(\R^n)$$
   for $q\in [1,2^*)$ and continuous embedding for $q\in[1, 2^*].$
\end{remark} 
\section{Fundamental properties of eigenvalues of $\mathcal{L}$}
In this section, we aim to establish several fundamental properties of the eigenvalues associated with \(\mathcal{L}\) in \eqref{1}.
Recalling $U$ is bounded and Proposition \ref{Poin}, we define $\lambda_1$ as 
\begin{equation*}
\lambda_1=\inf_{u\in \mathcal{X}^{1,2}_{\mathcal{D}}(U)\setminus \{0\}}  \frac{\int_{\Omega}|\nabla{u}|^2\,dx+ \int_{Q} \frac{|u(x)-u(y)|^{2}}{|x-y|^{n+2s}} \, dxdy} {\dint_{\Omega} | u|^2\,dx}\,.    
\end{equation*}
\
Equivalently, we can write  $\lambda_1 $ as
\begin{equation}\label{1ev}
\lambda_1=\inf_{ u\in \mathcal{X}^{1,2}_{\mathcal{D}}(U)\setminus\{0\},\, \|u\|^2_{L^2(\Omega)}=1}  \bigg(\int_{\Omega}|\nabla{u}|^2\,dx+ \int_{Q} \frac{|u(x)-u(y)|^{2}}{|x-y|^{n+2s}} \, dx dy\bigg).    
\end{equation}

\noindent\textbf{Proof of the Theorem \ref{prop3.1}:}
\begin{proof}
    The proof is established in Proposition 3.2, Lemma 3.3, and Proposition 3.4 as presented in \cite{JTL}. 
    \end{proof}
Let $\mathcal{J}: \mathcal{X}^{1,2}_{\mathcal{D}}(U)\to \R$ be a functional  defined by
\begin{equation}\label{eqJ}
\mathcal{J}(u):= \frac{1}{2}\int_{\Omega} |\nabla u|^2 \,dx +\frac{1}{2}\int_{Q} \frac{|u(x)-u(y)|^2}{|x-y|^{n+2s}} dxdy= \frac{1}{2}\eta(u)^2, ~\forall~ u\in \mathcal{X}^{1,2}_{\mathcal{D}}(U),
    \end{equation}which is convex and continuously differentiable, that is $C^1$, with its derivative given by
\begin{equation}\label{deri}
    \langle \mathcal{J}^{\prime}(u), v\rangle = \int_{\Omega} \nabla u\cdot\nabla v \,dx +\int_{Q} \frac{(u(x)-u(y))(v(x)-v(y))}{|x-y|^{n+2s}} dxdy, 
\end{equation}
for any $v\in \mathcal{X}^{1,2}_{\mathcal{D}}(U).$

\vspace{0.2em}
Using the minimization method, we shall present the following crucial lemma to establish Theorem \ref{p2.9} and Theorem \ref{p210} later.
\begin{Lemma}\label{lem3.2m}
 Let \( \mathcal{C} \) be a nonempty, weakly closed subspace of \( \mathcal{X}^{1,2}_{\mathcal{D}}(U) \), and define  
\[
\mathcal{G} = \{ u \in \mathcal{C} : \|u\|_{L^2(\Omega)} = 1 \}.
\]  
Then, there exists \( \hat{u} \in \mathcal{G} \) such that  
\begin{equation}\label{min}
\min_{u \in \mathcal{G}} \mathcal{J}(u) = \mathcal{J}(\hat{u}),
\end{equation}  
and  
\begin{equation}\label{min2}
\int_{\Omega} \nabla \hat{u} \cdot \nabla \varphi \, dx 
+ \int_{Q} \frac{(\hat{u}(x) - \hat{u}(y))(\varphi(x) - \varphi(y))}{|x-y|^{n+2s}} \, dx \, dy 
= \lambda^{\star} \int_{\Omega} \hat{u}(x) \varphi(x) \, dx,
\end{equation}  
for every \( \varphi \in \mathcal{C} \), where \( \lambda^{\star} = 2 \mathcal{J}(\hat{u}) > 0 \).  

    \end{Lemma}
\noindent
\begin{proof}
 First, we aim to prove \eqref{min}.  Suppose    \( \{ u_j \}_{j \in \mathbb{N}} \subset \mathcal{G} \), is a minimizing sequence such that
\begin{equation}\label{eq209}
\mathcal{J}(u_j) \to \inf_{u \in \mathcal{G}} \mathcal{J}(u) \geq 0 > -\infty \ \text{as} \ j \to +\infty. 
    \end{equation}
Then the sequence \( \{ \mathcal{J}(u_j) \}_{j \in \mathbb{N}} \) is bounded in \( \mathbb{R} \). Using the  definition of \( \mathcal{J} \), see \eqref{eqJ}, we get
\begin{equation}\label{eq2010}
\{ \eta(u_j) \}_{j \in \mathbb{N}} \ \text{is also bounded.} 
    \end{equation}
Since \(\mathcal{X}^{1,2}_{\mathcal{D}}(U)\) is a Hilbert space, up to a subsequence, still denoted by \(\{u_j\}_{j \in \mathbb{N}}\), such that \(\{u_j\}_{j \in \mathbb{N}}\) converges weakly in \(\mathcal{X}^{1,2}_{\mathcal{D}}(U)\) to some \(\hat{u} \in \mathcal{C}\), since \(\mathcal{C}\) is weakly closed. By the definition of weak convergence, we have  
\[
{\langle u_j, \varphi \rangle }_{\mathcal{X}^{1,2}_{\mathcal{D}}(U)} \to {\langle \hat{u}, \varphi \rangle }_{\mathcal{X}^{1,2}_{\mathcal{D}}(U)}
\]
for any \(\varphi \in \mathcal{X}^{1,2}_{\mathcal{D}}(U)\), as \(j \to +\infty\).
 Moreover
\begin{equation}\label{eq2.011}
\begin{cases}
u_j \rightharpoonup \hat{u} \mbox{ weakly in } \mathcal{X}^{1,2}_{{\mathcal{D}}}(U)\\
u_j \to \hat{u}\mbox{ strongly in } \ L^2 _{\mathrm{loc}}(\R^n),\\ 
u_j \to \hat{u} \mbox{ a.e in }  \R^n,
\end{cases}
    \end{equation}
as \( j \to +\infty \),  and \( \| \hat{u} \|_{L^2(\Omega)} = 1 \) by Remark \ref{r2.5}. Thus, \( \hat{u} \in \mathcal{G} \). By applying the weak lower semicontinuity of the norm in \(\mathcal{X}^{1,2}_{\mathcal{D}}(U)\), we obtain
\begin{align*}
\liminf_{j \to +\infty} \mathcal{J}(u_j) 
&= \frac{1}{2} \liminf_{j \to +\infty} \eta(u_j)^2 
\geq \frac{1}{2} \eta(\hat{u})^2=
 \mathcal{J}(\hat{u})
\geq \inf_{u \in \mathcal{G}} \mathcal{J}(u).
\end{align*}
Thus,  from \eqref{eq209}, it is easy to conclude \eqref{min} that is
\[
\mathcal{J}(\hat{u}) = \inf_{u \in \mathcal{G}} \mathcal{J}(u).
\]
We now proceed to prove \eqref{min2}. Let \( \varepsilon \in (-1,1) \), \( \varphi \in \mathcal{C} \) then we define  
\[
d_\varepsilon = \| \hat{u} + \varepsilon \varphi \|_{L^2(\Omega)} \quad \text{and} \quad u_\varepsilon = \frac{\hat{u} + \varepsilon \varphi}{d_\varepsilon}.
\]  
It is easy to verify that \( u_\varepsilon \in \mathcal{G} \).  Additionally, we have the following relations 
\[
d_\varepsilon^2 = \| \hat{u} \|_{L^2(\Omega)}^2 + 2\varepsilon \int_\Omega \hat{u}(x)\varphi(x) \, dx + o(\varepsilon),
\]  
and  
\[
\eta( \hat{u} + \varepsilon \varphi)^2 = \eta( \hat{u})^2 + 2\varepsilon \langle \hat{u}, \varphi \rangle_{\mathcal{X}^{1,2}_{\mathcal{D}}(U)} + o(\varepsilon).
\]
Since \( \| \hat{u} \|_{L^2(\Omega)} = 1 \), we can write  
\[
2 \mathcal{J}(u_\varepsilon) = \frac{\eta( \hat{u})^2 + 2\varepsilon \langle \hat{u}, \varphi \rangle_{\mathcal{X}^{1,2}_{\mathcal{D}}(U)} + o(\varepsilon)}{1 + 2\varepsilon \int_\Omega \hat{u}(x)\varphi(x) \, dx + o(\varepsilon)}.
\]
Expanding this, we obtain  
\[
2 \mathcal{J}(u_\varepsilon) = \left( 2 \mathcal{J}(\hat{u}) + 2\varepsilon \langle \hat{u}, \varphi \rangle_{\mathcal{X}^{1,2}_{\mathcal{D}}(U)} + o(\varepsilon) \right) \left( 1 - 2\varepsilon \int_\Omega \hat{u}(x)\varphi(x) \, dx + o(\varepsilon) \right).
\]
Simplifying further, we expand the functional \( \mathcal{J}(u_\varepsilon) \) as follows  
\begin{equation}\label{eqep}
2 \mathcal{J}(u_\varepsilon) = 2 \mathcal{J}(\hat{u}) + 2\varepsilon \left( \langle \hat{u}, \varphi \rangle_{\mathcal{X}^{1,2}_{\mathcal{D}}(U)} - 2 \mathcal{J}(\hat{u}) \int_\Omega \hat{u}(x)\varphi(x) \, dx \right) + o(\varepsilon).
\end{equation}
Now, using \eqref{eqep} and the minimality of \( \hat{u} \), i.e., 
\[
\mathcal{J}(\hat{u}) \leq \mathcal{J}(u_\varepsilon) \quad \text{for all } u_\varepsilon \in \mathcal{X}_{\mathcal{D}}^{1,2}(U),
\]
then we have
\[
2 \mathcal{J}(\hat{u}) \leq 2 \mathcal{J}(\hat{u}) + 2 \varepsilon \left( \langle \hat{u}, \varphi \rangle_{\mathcal{X}^{1,2}_{\mathcal{D}}(U)} - 2 \mathcal{J}(\hat{u}) \int_\Omega \hat{u}(x)\varphi(x) \, dx \right) + o(\varepsilon).
\]
Cancelling \( 2 \mathcal{J}(\hat{u}) \) from both sides and dividing through by \( \varepsilon \), for small \( \varepsilon \), we obtain
\[
\langle \hat{u}, \varphi \rangle_{\mathcal{X}^{1,2}_{\mathcal{D}}(U)} - 2 \mathcal{J}(\hat{u}) \int_\Omega \hat{u}(x) \varphi(x) \, dx \geq 0.
\]
Similarly, considering \( -\varepsilon \) (perturbing in the opposite direction), we deduce
\[
\langle \hat{u}, \varphi \rangle_{\mathcal{X}^{1,2}_{\mathcal{D}}(U)} - 2 \mathcal{J}(\hat{u}) \int_\Omega \hat{u}(x) \varphi(x) \, dx \leq 0.
\]
Thus, 
\[
\langle \hat{u}, \varphi \rangle_{\mathcal{X}^{1,2}_{\mathcal{D}}(U)} = 2 \mathcal{J}(\hat{u}) \int_\Omega \hat{u}(x) \varphi(x) \, dx.
\]
that is
\begin{equation*}
\int_\Omega \nabla \hat{u} \cdot \nabla \varphi \, dx + \int_Q \frac{(\hat{u}(x) - \hat{u}(y))(\varphi(x) - \varphi(y))}{|x-y|^{n+2s}} \, dx \, dy = 2 \mathcal{J}(\hat{u}) \int_\Omega \hat{u}(x) \varphi(x) \, dx,
\end{equation*}
which follows claim \eqref{min2}.
Note also that \( \mathcal{J}(\hat{u}) > 0 \), since \( \mathcal{J}(\hat{u}) = 0 \) would imply \( \hat{u} \equiv 0 \), which contradicts \( 0 \notin \mathcal{G} \). Thus, the proof is complete. 
\end{proof}
\begin{remark}\label{rem2.1}
   If $e$ is an eigenfunction of \eqref{dd2.2} corresponding to an eigenvalue $\lambda$ and choosing $\varphi=e$ in \eqref{dd2.2}, it follows that
    \begin{equation*}
        \int_{\Omega} |\nabla e|^2 \,dx + \int_{Q} \frac{|e(x) - e(y)|^2}{|x-y|^{n+2s}} \,dxdy = \lambda \|e\|_{L^2(\Omega)}^2.
    \end{equation*}
   \end{remark}
\noindent \textbf{Proof of Theorem \ref{p2.9}:}
 \begin{proof}
    \textbf{Proof of assertion $(1):$} 
  We define $\lambda_{k+1}$ as in \eqref{eq204}. It is easy to see that due to Lemma \ref{lem3.2m} the infimum in \eqref{eq204} exists and it is attained at some $e_{k+1} \in \mathcal{P}_{k+1}$, thanks to \eqref{min} and \eqref{min2}, applied here with $\mathcal{C}=\mathcal{P}_{k+1}$, which, by construction, is weakly closed. 
Furthermore, noting fact  $\mathcal{P}_{k+1} \subseteq \mathcal{P}_k \subseteq \mathcal{X}^{1,2}_{\mathcal{D}}(U)$,  we have 
\begin{equation}\label{eq211}
0<\lambda_1 \leqslant \lambda_2 \leqslant \ldots \leqslant \lambda_k \leqslant \lambda_{k+1} \leqslant \ldots
\end{equation}
First, we aim to show that 
\begin{equation}\label{eq212}
\lambda_1 \neq \lambda_2.
    \end{equation}
Indeed, if not, $e_2\in\mathcal{P}_2$ also would be an eigenfunction corresponding to $\lambda_1.$ Then using asseration (3) of Theorem \ref{prop3.1}, we can write $e_2=\zeta e_1$, with $\zeta \in \mathbb{R}$, and $\zeta \neq 0$ being $e_2 \not \equiv 0$. Hence, we obtain
$$
0=\left\langle e_1, e_2\right\rangle_{\mathcal{X}^{1,2}_{\mathcal{D}}(U)}=\zeta\eta(e_1)^2,
$$
that implies $e_1 \equiv 0$, and get a contradiction. Thus,  $\lambda_1\neq \lambda_2.$ From \eqref{eq211} and \eqref{eq212} we obtain \eqref{ineql}.
Also, by \eqref{min2} with $\mathcal{C}=\mathcal{P}_{k+1}$, then for all $\varphi~ \in \mathcal{P}_{k+1}$, we have
\begin{equation}\label{eq214}
\int_{\Omega} \nabla e_{k+1}\cdot \nabla \varphi\,dx+ \int_{Q}\frac{\left(e_{k+1}(x)-e_{k+1}(y)\right)(\varphi(x)-\varphi(y))}{|x-y|^{n+2s}}  d x d y 
=\lambda_{k+1} \int_{\Omega} e_{k+1}(x) \varphi(x) d x.
    \end{equation}
To establish that \( \lambda_{k+1} \) is an eigenvalue with eigenfunction \( e_{k+1} \), it is necessary to demonstrate that  
\begin{equation}\label{eqn216}  
\text{equation \eqref{eq214} holds for any } \varphi \in \mathcal{X}^{1,2}_{\mathcal{D}}(U), \text{ and not just for } \varphi \in \mathcal{P}_{k+1}.  
\end{equation}  
To establish this, we proceed with the induction method. Assuming that the claim holds for \( 1, \ldots, k \), we shall prove it for \( k+1 \). The initial case is ensured by the fact that \( \lambda_1 \) is an eigenvalue, as demonstrated in Proposition \ref{prop3.1}.  We utilize the direct sum decomposition  
\[
\mathcal{X}^{1,2}_{\mathcal{D}}(U) = \operatorname{span}\{e_1, \ldots, e_k\} \oplus \left(\operatorname{span}\{e_1, \ldots, e_k\}\right)^{\perp} = \operatorname{span}\{e_1, \ldots, e_k\} \oplus \mathcal{P}_{k+1},
\]  
where the orthogonal complement \( \perp \) is defined with respect to the scalar product of \( \mathcal{X}^{1,2}_{\mathcal{D}}(U) \), denoted by \( \langle \cdot, \cdot \rangle_{\mathcal{X}^{1,2}_{\mathcal{D}}(U)} \). Thus, given any \( \varphi \in \mathcal{X}^{1,2}_{\mathcal{D}}(U) \), we can decompose it as \( \varphi = \varphi_1 + \varphi_2 \), where \( \varphi_2 \in \mathcal{P}_{k+1} \) and  \[
\varphi_1 = \sum_{i=1}^k a_i e_i
\]  
for some coefficients \( a_1, \ldots, a_k \in \mathbb{R} \). By testing equation \eqref{eq214} with \( \varphi_2 = \varphi - \varphi_1 \), we  see that 
\begin{equation}\label{eq215}
\begin{aligned}
& \int_{\Omega} \nabla e_{k+1}\cdot \nabla \varphi\,dx+ \int_{Q}\frac{\left(e_{k+1}(x)-e_{k+1}(y)\right)(\varphi(x)-\varphi(y))}{|x-y|^{n+2s}}  \,d x d y-\lambda_{k+1} \int_{\Omega} e_{k+1}(x) \varphi(x) d x \\
& =\int_{\Omega} \nabla e_{k+1}\cdot \nabla \varphi_1\,dx+ \int_{Q}\frac{\left(e_{k+1}(x)-e_{k+1}(y)\right)(\varphi_1(x)-\varphi_1(y))}{|x-y|^{n+2s}}  \,d x d y-\lambda_{k+1} \int_{\Omega} e_{k+1}(x) \varphi_1(x) d x
\end{aligned}
    \end{equation}
$$
\begin{aligned}
&=\sum_{i=1}^k a_i\left[\int_{\Omega} \nabla e_{k+1}\cdot \nabla e_i\,dx+ \int_{Q}\frac{\left(e_{k+1}(x)-e_{k+1}(y)\right)(e_i(x)-e_i(y))}{|x-y|^{n+2s}}  \,d x d y\right. \left.-\lambda_{k+1} \int_{\Omega} e_{k+1}(x) e_i(x) d x\right].
\end{aligned}
$$
Moreover, by testing \eqref{dd2.2} for \( e_i \) against \( e_{k+1} \) for \( i = 1, \ldots, k \) (which is justified by the inductive assumption) and noting that \( e_{k+1} \in \mathcal{P}_{k+1} \), we observe that  
\[
0 = \int_{\Omega} \nabla e_{k+1} \cdot \nabla e_i \, dx + \int_{Q} \frac{\left(e_{k+1}(x) - e_{k+1}(y)\right)\left(e_i(x) - e_i(y)\right)}{|x - y|^{n+2s}} \, dx \, dy = \lambda_i \int_{\Omega} e_{k+1}(x) e_i(x) \, dx.
\]  
By \eqref{eq211}, for any \( i = 1, \ldots, k \), it follows that  
\[
\int_{\Omega} \nabla e_{k+1} \cdot \nabla e_i \, dx + \int_{Q} \frac{\left(e_{k+1}(x) - e_{k+1}(y)\right)\left(e_i(x) - e_i(y)\right)}{|x - y|^{n+2s}} \, dx \, dy = 0 = \int_{\Omega} e_{k+1}(x) e_i(x) \, dx.  
\]  
 Substituting this into \eqref{eq215}, we deduce that \eqref{eq214} holds for all \( \varphi \in \mathcal{X}^{1,2}_{\mathcal{D}}(U) \). Consequently, \( \lambda_{k+1} \) is an eigenvalue with eigenfunction \( e_{k+1} \).  

Next, we aim to prove that \eqref{eql206}. To do so, we first demonstrate that for \( k, m \in \mathbb{N} \) with \( k\neq m \), the following holds:  
\begin{equation}\label{eq216}  
\left\langle e_k, e_m \right\rangle_{\mathcal{X}^{1,2}_{\mathcal{D}}(U)} = 0 = \int_{\Omega} e_k(x) e_m(x) \, dx.  
\end{equation}  

Indeed, let \( k > m \), so \( k - 1 \geqslant m \). By the inclusion of orthogonal complements, we have  
\[
e_k \in \mathcal{P}_k = \left(\operatorname{span}\{e_1, \ldots, e_{k-1}\}\right)^{\perp} \subseteq \left(\operatorname{span}\{e_m\}\right)^{\perp}.  
\]  
As a result,  we have
\begin{equation}\label{eq217}  
\left\langle e_k, e_m \right\rangle_{\mathcal{X}^{1,2}_{\mathcal{D}}(U)} = 0.  
\end{equation}  
However, since \( e_k \) is an eigenfunction, we test \( e_k \) with \( \varphi = e_m \) in equation \eqref{dd2.2}, then we obtain  
\[
\int_{\Omega} \nabla e_{k} \cdot \nabla e_m \, dx + \int_{Q} \frac{\left(e_{k}(x) - e_{k}(y)\right)(e_m(x) - e_m(y))}{|x-y|^{n+2s}} \, dx \, dy  
= \lambda_k \int_{\Omega} e_k(x) e_m(x) \, dx.
\]
From this and equation \eqref{eq217}, we derive \eqref{eq216}. To complete the proof of \eqref{eql206}, assume, for the sake of contradiction, that \( \lambda_k \to a \) for some constant \( a \in \mathbb{R} \), implying \( \lambda_k \) is bounded in \( \mathbb{R} \). Since \( \eta(e_k)^2 = \lambda_k \), using Remark \ref{rem2.1}, and applying the embedding result (see Remark \ref{r2.5}), we deduce that there exists a subsequence for which  
\[
e_{k_j} \to e_{\infty} \quad \text{in } L^2_{\text{loc}}(\mathbb{R}^n)
\]  
as \( k_j \to +\infty \), for some \( e_{\infty} \in L^2_{\text{loc}}(\mathbb{R}^n) \). In particular,  
\begin{equation} \label{eq218}
e_{k_j} \text{ is a Cauchy sequence in } L^2_{\text{loc}}(\mathbb{R}^n).
\end{equation}
However, by \eqref{eq216}, \( e_{k_j} \) and \( e_{k_i} \) are orthogonal in \( L^2_{\text{loc}}(\mathbb{R}^n) \), leading to  
\[
\left\|e_{k_j} - e_{k_i}\right\|_{L^2(\Omega)}^2 = \left\|e_{k_j}\right\|_{L^2(\Omega)}^2 + \left\|e_{k_i}\right\|_{L^2(\Omega)}^2 = 2.
\]
Since this contradicts \eqref{eq218}, we have established the claim \eqref{eql206}.  
To complete the proof of assertion (1), it remains to show that the sequence of eigenvalues constructed in \eqref{eq204} exhausts all the eigenvalues of the problem, i.e., every eigenvalue of problem \eqref{dd2.2} can be expressed in the form \eqref{eq204}. We proceed by contradiction.  
Assume there exists an eigenvalue  
\begin{equation} \label{eq220}
\lambda \notin \left\{\lambda_k\right\}_{k \in \mathbb{N}},
\end{equation}  
and let \( e \in \mathcal{X}^{1,2}_{\mathcal{D}}(U) \) be an eigenfunction associated with \( \lambda \), normalized such that \( \|e\|_{L^2(\Omega)} = 1 \). Using Remark \ref{rem2.1}, we have  
\begin{equation} \label{2J}
2 \mathcal{J}(e) = \int_{\Omega} |\nabla e(x)|^2 \, dx + \int_{Q} \frac{\left|e(x) - e(y)\right|^2}{|x-y|^{n+2s}} \, dx \, dy = \lambda.
\end{equation}  
From the definition of \( \lambda_1 \) as the infimum given in \eqref{1ev},  
\[
\lambda_1 = \int_{\Omega} |\nabla e_1|^2 \, dx + \int_{Q} \frac{|e_1(x) - e_1(y)|^2}{|x-y|^{n+2s}} \, dx \, dy,
\]  
it follows that  
\[
\lambda = 2 \mathcal{J}(e) \geq 2 \mathcal{J}\left(e_1\right) = \lambda_1.
\]  
This, \eqref{eq220} and \eqref{eql206} imply that there exists $k \in \mathbb{N}$ such that
\begin{equation}\label{eq222}
\lambda_k<\lambda<\lambda_{k+1} .
    \end{equation}
Now, we claim that  
\begin{equation} \label{eq223}
e \notin \mathcal{P}_{k+1}.
\end{equation}  
To see this, assume \( e \in \mathcal{P}_{k+1} \). Then, using \eqref{2J} and \eqref{eq204}, we obtain that  
\[
\lambda = 2 \mathcal{J}(e) \geq \lambda_{k+1}.
\]  
This contradicts \eqref{eq222}, thereby proving \eqref{eq223}.  
As a consequence of \eqref{eq223}, there exists some \( i \in \{1, \ldots, k\} \) such that  
\[
\left\langle e, e_i \right\rangle_{\mathcal{X}^{1,2}_{\mathcal{D}}(U)} \neq 0.
\]  
However, this contradicts with [\cite{JTL}, Lemma 3.6], thereby proving that \eqref{eq220} is not true. Hence, all eigenvalues belong to the sequence \( \{\lambda_k\}_{k \in \mathbb{N}} \),  
  assertion (1) follows.\\

\textbf{Proof of assertion $(2):$} 
By applying \eqref{min} with \( \mathcal{C} = \mathcal{P}_{k+1} \), the minimum that defines \( \lambda_{k+1} \) is achieved by some \( e_{k+1} \in \mathcal{P}_{k+1} \).  The fact that \( e_{k+1} \) is an eigenfunction corresponding to \( \lambda_{k+1} \) was established in \eqref{eqn216}, and \eqref{eq2011} follows directly from \eqref{min2}.
 \end{proof}

\noindent \textbf{Proof of Theorem \ref{p210}:}
 \begin{proof}
The orthogonality stated follows directly from \eqref{eq216}. To complete the proof of this proposition, it remains to show that the sequence of eigenfunctions \( \{e_k\}_{k \in \mathbb{N}} \) forms a basis for both \( L^2(\Omega) \) and \( \mathcal{X}^{1,2}_{\mathcal{D}}(U) \).  
We begin by proving that it is a basis for \( \mathcal{X}^{1,2}_{\mathcal{D}}(U) \). Specifically, we aim to demonstrate that  
\begin{equation}\label{eq340}  
\text{if } w \in \mathcal{X}^{1,2}_{\mathcal{D}}(U) \text{ satisfies } \langle w, e_k \rangle_{\mathcal{X}^{1,2}_{\mathcal{D}}(U)} = 0 \text{ for all } k \in \mathbb{N}, \text{ then } w \equiv 0.  
\end{equation}  
We proceed by contradiction and assume there exists a nontrivial \( w\in \mathcal{X}^{1,2}_{\mathcal{D}}(U) \) such that  
\begin{equation}\label{eq341}  
\langle w, e_k \rangle_{\mathcal{X}^{1,2}_{\mathcal{D}}(U)} = 0 \quad \text{for all } k \in \mathbb{N}.  
\end{equation}  
Without loss of generality, we normalize \( w \) such that \( \|w\|_{L^2(\Omega)} = 1 \). From \eqref{eql206}, there exists \( k \in \mathbb{N} \) such that  
\[
2 \, \mathcal{J}(w) < \lambda_{k+1} = \min_{u \in \mathcal{P}_{k+1}, \|u\|_{L^2(\Omega)}=1} \left\{\int_{\Omega} |\nabla u|^2 \, dx + \int_{Q} \frac{|u(x) - u(y)|^2}{|x-y|^{n+2s}} \, dx \, dy\right\}.  
\]
Thus, \( w\notin \mathcal{P}_{k+1} \), which implies there exists \( j \in \mathbb{N} \) such that \( \langle w, e_j \rangle_{\mathcal{X}^{1,2}_{\mathcal{D}}(U)} \neq 0 \). This contradicts \eqref{eq341}, thereby proving \eqref{eq340}.  

Now, using a standard Fourier analysis argument, we now show that \( \{e_k\}_{k \in \mathbb{N}} \) forms a basis for \( \mathcal{X}^{1,2}_{\mathcal{D}}(U) \). We define \( M_i = \frac{e_i}{\eta(e_i)} \), and for any \( f \in \mathcal{X}^{1,2}_{\mathcal{D}}(U) \),   
\[
f_j = \sum_{i=1}^j \langle f, M_i \rangle_{\mathcal{X}^{1,2}_{\mathcal{D}}(U)} M_i.  
\]  
Note that, for any \( j \in \mathbb{N} \),  
\begin{equation}\label{eq342}  
f_j \in \operatorname{span}\{e_1, \dots, e_j\}.  
\end{equation}  
Define \( w_j = f - f_j \). Using the orthogonality of \( \{e_k\}_{k \in \mathbb{N}} \) in \( \mathcal{X}^{1,2}_{\mathcal{D}}(U) \), we have  
\begin{equation}\label{eq3032}
0 \leq \eta(w_j)^2 = \langle w_j, w_j \rangle_{\mathcal{X}^{1,2}_{\mathcal{D}}(U)} = \langle f - f_j, f - f_j \rangle_{\mathcal{X}^{1,2}_{\mathcal{D}}(U)}= \eta(f)^2 + \eta(f_j)^2 - 2 \langle f, f_j \rangle_{\mathcal{X}^{1,2}_{\mathcal{D}}(U)}     
\end{equation} 
Substituting the definition of \( f_j \), we obtain  
\[
\eta(w_j)^2 = \eta(f)^2 + \langle f_j, f_j \rangle_{\mathcal{X}^{1,2}_{\mathcal{D}}(U)} - 2 \sum_{i=1}^j \langle f, M_i \rangle^2_{\mathcal{X}^{1,2}_{\mathcal{D}}(U)},  
\]  
simplifying further, we get 
\[
\eta(w_j)^2 = \eta(f)^2 - \sum_{i=1}^j \langle f, M_i \rangle^2_{\mathcal{X}^{1,2}_{\mathcal{D}}(U)}.  
\]  
Thus, for any \( j \in \mathbb{N} \), putting this in \eqref{eq3032}, we conclude  
\[
\sum_{i=1}^j \langle f, M_i \rangle^2_{\mathcal{X}^{1,2}_{\mathcal{D}}(U)} \leq \eta(f)^2.  
\]  
This shows that the series  
\[
\sum_{i=1}^{+\infty} \langle f, M_i \rangle^2_{\mathcal{X}^{1,2}_{\mathcal{D}}(U)}  
            ~~~\text{is convergent}.
            \]
 Thus, if we define  
\(
\tau_j = \sum_{i=1}^j \langle f, M_i \rangle^2_{\mathcal{X}^{1,2}_{\mathcal{D}}(U)},  
\)  
we have that  
\begin{equation}\label{eq343}
\{\tau_j\}_{j \in \mathbb{N}} \text{ is a Cauchy sequence in } \mathbb{R}.
\end{equation}  
Using the orthogonality of \( \{e_k\}_{k \in \mathbb{N}} \) in \( \mathcal{X}^{1,2}_{\mathcal{D}}(U) \), consider \( l > j \) to get 
\begin{align*}
w_l - w_j = \sum_{i=j+1}^l \langle f, M_i \rangle_{\mathcal{X}^{1,2}_{\mathcal{D}}(U)} M_i ~\implies~ \eta(w_l - w_j)^2 &=\eta\left( \sum_{i=j+1}^l \langle f, M_i \rangle_{\mathcal{X}^{1,2}_{\mathcal{D}}(U)} M_i \right)^2\\
&= \sum_{i=j+1}^l \langle f, M_i \rangle^2_{\mathcal{X}^{1,2}_{\mathcal{D}}(U)}
= \tau_l - \tau_j.
    \end{align*}  
From above and recalling \eqref{eq343}, we conclude that
\( \{w_j\}_{j \in \mathbb{N}} \) forms a Cauchy sequence in \( \mathcal{X}^{1,2}_{\mathcal{D}}(U) \).  
Since \( \mathcal{X}^{1,2}_{\mathcal{D}}(U) \) is a Hilbert space, there exists \( w\in \mathcal{X}^{1,2}_{\mathcal{D}}(U) \) such that  
\begin{equation}\label{eq344}
w_j \to w \text{ in } \mathcal{X}^{1,2}_{\mathcal{D}}(U) \text{ as } j \to +\infty.
\end{equation}   
For \( j \geq k \), observe that  
\[
\langle w_j, M_k \rangle_{\mathcal{X}^{1,2}_{\mathcal{D}}(U)} = \langle f, M_k \rangle_{\mathcal{X}^{1,2}_{\mathcal{D}}(U)} - \langle f_j, M_k \rangle_{\mathcal{X}^{1,2}_{\mathcal{D}}(U)}.
\]  
Since \( f_j = \sum_{i=1}^j \langle f, M_i \rangle_{\mathcal{X}^{1,2}_{\mathcal{D}}(U)} M_i \), it follows that  
\[
\langle f_j, M_k \rangle_{\mathcal{X}^{1,2}_{\mathcal{D}}(U)} = \langle f, M_k \rangle_{\mathcal{X}^{1,2}_{\mathcal{D}}(U)} \quad \text{for } k \leq j.
\]  
Hence,  
\(
\langle w_j, M_k \rangle_{\mathcal{X}^{1,2}_{\mathcal{D}}(U)} = 0.
\)  
Thus, by \eqref{eq344}, it easily follows that $\langle w, M_k \rangle_{\mathcal{X}^{1,2}_{\mathcal{D}}(U)} = 0$, for any $k \in \mathbb{N}$, so, by \eqref{eq340}, we get $w= 0$. So, we have
\[
f_j = f - w_j \to f - w = f \text{ in } \mathcal{X}^{1,2}_{\mathcal{D}}(U) ~ \text{as}~j \to +\infty.
\]
This and \eqref{eq342} yield that $\{e_k\}_{k \in \mathbb{N}}$ is a basis in $\mathcal{X}^{1,2}_{\mathcal{D}}(U)$.

To conclude the proof of this proposition, it remains to verify that \(\{e_k\}_{k \in \mathbb{N}}\) forms a basis for \(L^2(\Omega)\). Consider any \(w \in L^2(\Omega)\), and let a sequence \(\{w_j\}_{j \in \mathbb{N}} \subseteq C^{\infty}_0(U)\) be such that \(\|w_j - w\|_{L^2(\Omega)} \leq 1/j\). It is easy to see that \(w_j \in \mathcal{X}^{1,2}_{\mathcal{D}}(U)\).  Since \(\{e_k\}_{k \in \mathbb{N}}\) is already established as a basis for \(\mathcal{X}^{1,2}_{\mathcal{D}}(U)\), there exists \(k_j \in \mathbb{N}\) and a function \(v_j \in \text{span}\{e_1, \ldots, e_{k_j}\}\) such that  
\[
\eta(w_j - v_j) \leq \frac{1}{j}.
\]  
Using Remark \(\ref{r2.5}\), it follows that  
\[
\|w_j - v_j\|_{L^2(\Omega)} \leq C \eta(w_j - v_j) \leq \frac{C}{j},
\]  
where \(C = C(n, s, U) > 0\) is the embedding constant.  
By using the triangle inequality, it is easy to see that  
\[
\|w - v_j\|_{L^2(\Omega)} \leq \|w - w_j\|_{L^2(\Omega)} + \|w_j - v_j\|_{L^2(\Omega)} \leq \frac{C+1}{j}.
\]  
This demonstrates that the eigenfunctions \(\{e_k\}_{k \in \mathbb{N}}\) of \eqref{1} form a basis for \(L^2(\Omega)\). Hence, the proof is complete.
 \end{proof}
\begin{Proposition}
    Each eigenvalue \(\lambda_k\) has a finite multiplicity. Specifically, if an eigenvalue \(\lambda_k\) satisfies the condition  
\begin{equation}\label{eq230}
    \lambda_{k-1} < \lambda_k = \cdots = \lambda_{k+h} < \lambda_{k+h+1},
\end{equation}  
for some \(h \in \mathbb{N}_0\), then the set of all eigenfunctions associated with \(\lambda_k\) is given by  
\[
\text{span}\{e_k, e_{k+1}, \ldots, e_{k+h}\}.
\]  
\end{Proposition}
\begin{proof}
 Let \(h \in \mathbb{N}_0\), such that \eqref{eq230} holds. From Theorem \ref{p210}, we know that every element of \(\operatorname{span}\{e_k, \ldots, e_{k+h}\}\) is an eigenfunction of problem \eqref{dd2.2} corresponding to the eigenvalue \(\lambda_k = \cdots = \lambda_{k+h}\). Therefore, it remains to prove that any nonzero eigenfunction \(\phi \not\equiv 0\) associated with \(\lambda_k\) belongs to \(\operatorname{span}\{e_k, \ldots, e_{k+h}\}\).  

To establish this, we decompose  
\[
\mathcal{X}^{1,2}_{\mathcal{D}}(U) = \operatorname{span}\{e_k, \ldots, e_{k+h}\} \oplus \left(\operatorname{span}\{e_k, \ldots, e_{k+h}\}\right)^{\perp}.
\]  
Thus, we can express \(\phi\) as  
\[
\phi = \phi_1 + \phi_2,  
\]  
where  
\begin{equation}\label{eq2031}
\phi_1 \in \operatorname{span}\{e_k, \ldots, e_{k+h}\} \quad \text{and} \quad \phi_2 \in \left(\operatorname{span}\{e_k, \ldots, e_{k+h}\}\right)^{\perp}.
\end{equation}  
In particular, orthogonality implies  
\begin{equation}\label{eq231}
\left\langle \phi_1, \phi_2 \right\rangle_{\mathcal{X}^{1,2}_{\mathcal{D}}(U)} = 0.
\end{equation}  
Since \(\phi\) is an eigenfunction corresponding to \(\lambda_k\), we may write \eqref{dd2.2} and test it against itself. This yields  
\[
\lambda_k \|\phi\|_{L^2(\Omega)}^2 = \eta(\phi)^2 = \eta(\phi_1)^2 + \eta(\phi_2)^2,
\]  
where the equality follows from \eqref{eq231}. 

Furthermore, by Theorem \ref{p210}, the functions \(e_k, \ldots, e_{k+h}\) are eigenfunctions associated with the eigenvalue \(\lambda_k = \cdots = \lambda_{k+h}\). Consequently,  
\begin{equation}\label{eq232}
\phi_1 \text{ is also an eigenfunction corresponding to } \lambda_k.
\end{equation}  
As a result, we can write \eqref{dd2.2} for \(\phi_1\) and test it against \(\phi_2\). Recalling \eqref{eq231}, this gives  
\begin{equation}\label{eq234}
\lambda_k \int_{\Omega} \phi_1(x) \phi_2(x) \, dx = \left\langle \phi_1, \phi_2 \right\rangle_{\mathcal{X}^{1,2}_{\mathcal{D}}(U)} = 0,  
\end{equation}  
which implies  
\[
\int_{\Omega} \phi_1(x) \phi_2(x) \, dx = 0.
\]  
Thus, we have  
\begin{equation}\label{eq235}
\|\phi\|_{L^2(\Omega)}^2 = \|\phi_1 + \phi_2\|_{L^2(\Omega)}^2 = \|\phi_1\|_{L^2(\Omega)}^2 + \|\phi_2\|_{L^2(\Omega)}^2.
\end{equation}  
Now, we can express \(\phi_1\) as  
\[
\phi_1 = \sum_{i=k}^{k+h} a_i e_i,
\]  
where \(a_i \in \mathbb{R}\). Using the orthogonality (see Theorem \ref{p210} and \eqref{eq2011}), we compute 
\begin{equation}\label{eq236}
\eta(\phi_1)^2 = \eta\left(\sum_{i=k}^{k+h} a_i e_i \right)^2 = \sum_{i=k}^{k+h} a_i^2 \eta(e_i)^2 = \sum_{i=k}^{k+h} a_i^2 \lambda_i = \lambda_k \sum_{i=k}^{k+h} a_i^2 = \lambda_k \|\phi_1\|_{L^2(\Omega)}^2.
\end{equation}  
Using \eqref{eq232}, and noting that \(\phi\) is an eigenfunction associated with \(\lambda_k\), we conclude that \(\phi_2\) is also an eigenfunction corresponding to \(\lambda_k\). Therefore, recalling \eqref{eq230} and using [\cite{JTL}, Lemma 3.6], we deduce 
\[
\left\langle \phi_2, e_1 \right\rangle_{\mathcal{X}^{1,2}_{\mathcal{D}}(U)} = \cdots = \left\langle \phi_2, e_{k-1} \right\rangle_{\mathcal{X}^{1,2}_{\mathcal{D}}(U)} = 0.
\]  
This, together with \eqref{eq2031}, implies  
\begin{equation}\label{eq2034}
\phi_2 \in \left(\operatorname{span}\{e_1, \ldots, e_{k+h}\}\right)^{\perp} = \mathcal{P}_{k+h+1}.
\end{equation}  
Now, we aim to prove that  
\begin{equation}\label{eq2038}
\phi_2 \equiv 0.
\end{equation}  
To prove this, assume by contradiction that \(\phi_2 \not\equiv 0\). Using \eqref{eq206} and \eqref{eq2034}, we write  
\begin{equation}\label{eq238}
\begin{aligned}
\lambda_k < \lambda_{k+h+1} &= \min_{u \in \mathcal{P}_{k+h+1} \setminus \{0\}} \frac{\int_{\Omega} |\nabla u|^2 \, dx + \int_{Q} \frac{|u(x) - u(y)|^2 \, dx \, dy}{|x-y|^{n-2s}}}{\int_{\Omega} |u(x)|^2 \, dx} \\
&\leq \frac{\int_{\Omega} |\nabla \phi_2|^2 \, dx + \int_{Q} \frac{|\phi_2(x) - \phi_2(y)|^2 \, dx \, dy}{|x-y|^{n-2s}}}{\int_{\Omega} |\phi_2(x)|^2 \, dx} = \frac{\eta(\phi_2)^2}{\|\phi_2\|_{L^2(\Omega)}^2}.
\end{aligned}
\end{equation}  
From using \eqref{eq234}, \eqref{eq235}, \eqref{eq236}, and \eqref{eq238} then we obtain that
$$
\begin{aligned}
\lambda_k\|\phi\|_{L^2(\Omega)}^2  =\eta(\phi_1)^2+\eta(\phi_2)^2  >\lambda_k\left\|\phi_1\right\|_{L^2(\Omega)}^2+\lambda_k\left\|\phi_2\right\|_{L^2(\Omega)}^2 =\lambda_k\|\phi\|_{L^2(\Omega)}^2
\end{aligned}
$$
which is a contradiction. Thus, the claim in \eqref{eq2038} is established. Combining \eqref{eq2031} and \eqref{eq2038}, we conclude that  
\[
\phi = \phi_1 \in \operatorname{span}\{e_k, \ldots, e_{k+h}\},
\]  
as required. This completes the proof of the proposition.
\end{proof}
The following remark emphasizes key aspects of variational inequalities.
\begin{remark}  
\begin{enumerate}  
    \item For any \( u \in \mathrm{span}(u_1, \ldots, u_k)^\perp = \mathcal{P}_{k+1} \), the following inequality holds true:  
    \[  
    \lambda_{k+1} \int_{\Omega} u^2 \, dx \leq \int_{\Omega} |\nabla u|^2 \, dx + \int_{Q} \frac{|u(x) - u(y)|^2}{|x-y|^{n+2s}} \, dx \, dy.  
    \]  

    \item For any \( u \in \mathrm{span}(u_1, \ldots, u_k) \), the following inequality holds true:  
    \[  
    \int_{\Omega} |\nabla u|^2 \, dx + \int_{Q} \frac{|u(x) - u(y)|^2}{|x-y|^{n+2s}} \, dx \, dy \leq \lambda_k \int_{\Omega} u^2 \, dx.  
    \]  
\end{enumerate}  
\end{remark}

In Theorem \ref{p2.9}, we presented a variational characterization of \(\lambda_k\) using the subspace \(\mathcal{P}_{k+1}\). In the next result, we establish a different variational description of the eigenvalues of \(\mathcal{L}\), formulated within the context of a finite-dimensional space.

\begin{Lemma}
Let \(\{\lambda_k\}_{k \in \mathbb{N}}\) denote the sequence of eigenvalues associated of problem \eqref{dd2.2}, ordered as,  
\[
0 < \lambda_1 < \lambda_2 \leq \cdots \leq \lambda_k \leq \lambda_{k+1} \leq \ldots,
\]  
and let \(\{e_k\}_{k \in \mathbb{N}}\) represent the sequence of eigenfunctions corresponding to these eigenvalues. For any \(k \in \mathbb{N}\), the eigenvalue \(\lambda_k\) can be characterized by the following variational formula  
\[
\lambda_k = \max_{u \in \operatorname{Span}\{e_1, \ldots, e_k\} \setminus \{0\}} \frac{\int_{\Omega} |\nabla u(x)|^2 \, dx + \int_{Q} \frac{|u(x) - u(y)|^2}{|x-y|^{n-2s}} \, dx \, dy}{\int_{\Omega} |u(x)|^2 \, dx}.
\]
    \end{Lemma}
\begin{proof}
 Let \(k \in \mathbb{N}\). Since \(\lambda_k\) is the eigenvalue corresponding to the eigenfunction \(e_k\), we have
\begin{equation}\label{lambda_k}
\begin{aligned}
\lambda_k & = \frac{\int_{\Omega} |\nabla e_k(x)|^2 \, dx + \int_{Q} \frac{|e_k(x) - e_k(y)|^2}{|x-y|^{n-2s}} \, dx \, dy}{\int_{\Omega} |e_k(x)|^2 \, dx} \\
& \leq \max_{u \in \operatorname{Span}\{e_1, \ldots, e_k\} \setminus \{0\}} \frac{\int_{\Omega} |\nabla u(x)|^2 \, dx + \int_{Q} \frac{|u(x) - u(y)|^2}{|x-y|^{n-2s}} \, dx \, dy}{\int_{\Omega} |u(x)|^2 \, dx}.
\end{aligned}
\end{equation}
Now, consider any \(u \in \operatorname{Span}\{e_1, \ldots, e_k\} \setminus \{0\}\). We can express \(u\) as  
\[
u = \sum_{i=1}^k u_i e_i, \quad \text{with }~ \sum_{i=1}^k u_i^2 \neq 0.
\]
For such \(u\), we have  
\[
\frac{\int_{\Omega} |\nabla u(x)|^2 \, dx + \int_{Q} \frac{|u(x) - u(y)|^2}{|x-y|^{n-2s}} \, dx \, dy}{\int_{\Omega} |u(x)|^2 \, dx} = \frac{\sum_{i=1}^k u_i^2 \eta(e_i)}{\sum_{i=1}^k u_i^2} = \frac{\sum_{i=1}^k u_i^2 \lambda_i}{\sum_{i=1}^k u_i^2}.
\]
Since \(\{e_1, \ldots, e_k\}\) are orthonormal in \(L^2(\Omega)\) and orthogonal in \(\mathcal{X}^{1,2}_{\mathcal{D}}(U)\) (as stated in Theorem \ref{p210}), and since \(\lambda_k \geq \lambda_i\) for all \(i = 1, \ldots, k\), it follows that  
\[
\frac{\sum_{i=1}^k u_i^2 \lambda_i}{\sum_{i=1}^k u_i^2} \leq \lambda_k.
\]
Taking the maximum over \(u \in \operatorname{Span}\{e_1, \ldots, e_k\} \setminus \{0\}\), we obtain  
\[
\max_{u \in \operatorname{Span}\{e_1, \ldots, e_k\} \setminus \{0\}} \frac{\int_{\Omega} |\nabla u(x)|^2 \, dx + \int_{Q} \frac{|u(x) - u(y)|^2}{|x-y|^{n-2s}} \, dx \, dy}{\int_{\Omega} |u(x)|^2 \, dx} \leq \lambda_k.
\]
We conclude that the assertion holds by combining this with \eqref{lambda_k}.
    \end{proof}
\section*{Acknowledgements} 
 Lovelesh Sharma received assistance from the UGC Grant with reference no. 191620169606 funded by the Government
of India.
\vspace{0.5cm}

\section*{Data availability} Data sharing not applicable to this article as no datasets were generated or analysed during the current study.

\section*{Competing interests}
The authors declare that there is no competing interest between them.

\end{document}